\documentclass[12pt, final]{amsart}

\usepackage{biblatex}
\usepackage{hyperref}
\usepackage{extpfeil}

\usepackage{tabularray}

\usepackage{enumerate}
\usepackage{amssymb,amsfonts,amsmath,amsthm}
\usepackage{tikz}
\usepackage{tikz-cd}
\usepackage{showkeys}
\usepackage{tcolorbox}
\usepackage{xcolor}
\usepackage{framed}
\usepackage{todonotes}
\usepackage{cancel}
\usepackage{pifont}

\definecolor{shadecolor}{rgb}{0.3,0.7,0.9}

\tikzset{node distance=2cm, auto}

\theoremstyle{remark}
\newtheorem{example}{Example}[section]
\newtheorem{remark}[example]{Remark}

\theoremstyle{definition}
\newtheorem{definition}[example]{Definition}
\theoremstyle{plain}
\newtheorem{proposition}[example]{Proposition}
\newtheorem{corollary}[example]{Corollary}

\newtheorem{theorem}[example]{Theorem}
\newtheorem{lemma}[example]{Lemma}

\DeclareFieldFormat{postnote}{#1}
\DeclareFieldFormat{multipostnote}{#1}

\newcommand{\etale}{\'{e}tale}

\newcommand{\ZZ}{\mathbb{Z}}
\newcommand{\QQ}{\mathbb{Q}}

\renewcommand{\AA}{\mathbb{A}}

\newcommand{\sC}{\mathcal{C}}

\newcommand{\sO}{\mathcal{O}}

\newcommand{\fX}{{\mathfrak X}}
\newcommand{\fY}{{\mathfrak Y}}

\newcommand{\tX}{\tilde{X}}

\newcommand{\ftX}{\tilde{\mathfrak X}}

\newcommand{\spec}{{\rm Spec}}

\renewcommand{\sl}{{\rm SL}}

\newcommand{\et}{{\acute{e}t}}
\DeclareMathOperator{\aff}{Aff}
\DeclareMathOperator{\sch}{Sch}

\DeclareMathOperator{\fint}{fin.type}
\renewcommand{\top}{{\rm Top}}

\DeclareMathOperator{\fin}{fin}

\DeclareMathOperator{\bA}{{\bf A}}
\newcommand{\cS}{\cancel{S}}

\DeclareMathOperator{\gl}{GL}

\DeclareMathOperator{\Spec}{Spec}

\DeclareMathOperator{\Br}{Br}
\DeclareMathOperator{\inv}{inv}

\newcommand{\Gm}{{\mathbb{G}_{\mathrm{m}}}}

\newextarrow{\xbigtoto}{{20}{20}{20}{20}}
   {\bigRelbar\bigRelbar{\bigtwoarrowsleft\rightarrow\rightarrow}}

   \tikzset{
    labl/.style={anchor=south, rotate=90, inner sep=.5mm}
}

\makeatletter
\newcommand{\colim@}[2]{%
  \vtop{\m@th\ialign{##\cr
    \hfil$#1\operator@font colim$\hfil\cr
    \noalign{\nointerlineskip\kern1.5\ex@}#2\cr
    \noalign{\nointerlineskip\kern-\ex@}\cr}}%
}
\newcommand{\colim}{%
  \mathop{\mathpalette\colim@{\rightarrowfill@\textstyle}}\nmlimits@
}
\makeatother

\title[Classifying stacks]{ Approximation theorems for classifying stacks over number fields}

\author[A. Dhillon]{Ajneet Dhillon}
\email{adhill3@uwo.ca}

\begin{document}

\begin{abstract}
    Approximation theorems for algebraic stacks over a number field $k$ are studied in this article. For $G$ a connected linear algebraic 
    group over a number field we prove strong approximation with Brauer-Manin obstruction for the classifying stack $BG$. This result answers a very concrete
    question, given $G$-torsors $P_{v}$ over $k_{v}$, where $v$ ranges over a finite number of places, when can you approximate the $P_{v}$
    by a $G$-torsor $P$ defined over $k$. 
\end{abstract}

\maketitle
\section{Introduction}
This paper is motivated by the following question,
let $G$ be a  connected linear algebraic group over a number field $k$ and let $v_{1}, \ldots , v_{n}$ be some non-archimedean places of $k$. 
Denote by $k_{v_i}$ the completion of $k$ at $v_{i}$. 
Fix $G$-torsors $P_i$  over  each $k_{v_i}$. Given this data, can one find a $G$-torsor $P$ over $k$ that specialises to 
each of the fixed torsors?
We provide a condition (vanishing of the Brauer-Manin obstruction) on the $P_i$ in order to guarantee the existence of a $P$ that is $v_{i}$-adically close to $P_{i}$. In 
this context, the notion of being $v_i$-adically close is defined in sections 2.3 and 2.4. 
It is worth remarking here that Lang's theorem implies that every $G$-torsor over $\sO_{v_i}$ is trivial where $\sO_{v_i}$ is the completion of 
the ring of integers at the place non-archimedean place $v_i$ so that this question is a version of the strong approximation property for the classifying stack of $G$-torsors. 
Note that for an algebraic group, being connected and geometrically connected are the same, as the identity component will be preserved by Galois actions.

We let $BG$ be the classifying stack of $G$-torsors and let $\bA_{k}$ be the adele ring of $k$. In section \ref{s:topology}, we equip the adelic points of $BG$ with the structure
of a topological space. The essential idea for this construction can be found in \cite{christensen2020}. 
There is a map $BG(k)\to BG(\bA_{k})$ induced from the natural embedding $k\to \bA_{k}$, here $BG(k)$ (resp. $BG(\bA_k)$) is the set, not groupoid of $k$-points
(resp. $\bA_k$-points) of $BG$. 
This map need not be surjective nor injective, as discussed in the next paragraph. 
Our results compute the closure of the image of $BG(k)$ inside $BG(\bA_{k})$, see Theorem \ref{t:sApproxBG}.
A variant of this theorem that applies to quotient stacks is produced in Theorem \ref{t:roy}. The main source of examples to which this theorem applies are 
generated by quotients of groupic varieties, see \cite{cao2018} and the discussion after Theorem \ref{t:roy}. Note however, Theorem \ref{t:sApproxBG} is not a 
corollary of Theorem \ref{t:roy}. 

The map $BG(k)\to BG(\bA_{k})$ has been studied in classically in a different guise, as a map on Galois cohomology. For ${\rm PGL}_n$ the map 
need not be surjective as class field theory describes the Brauer group of a number field in terms of local Brauer groups and hence puts a restriction
upon which ${\rm PGL}_n$-torsors can lift. The non-injectivity of the map is rather subtle and it often is injective. The injectivity is known as 
the Hasse principle and it may fail for some groups. For a general discussion see \cite[pg. 285]{rapinchuk}.

Given a scheme $X$ of finite type over $k$, it is well known how to equip $X(k_{v})$ with the structure of a topological space, see \cite{conrad2020}. 
The insight of \cite{christensen2020} is if $\fX$ is a finite type algebraic stack and $X\to \fX$ a smooth surjective presentation then $\fX(k_{v})$
should inherit the quotient topology from $X(k_{v})$ if the properties of the topologization on schemes are to extend to algebraic stacks. 
The paper \cite{christensen2020} provides a topologization on schemes over fields or rings more general than just a number field or a global field. 
Unfortunately, in topologizing the adelic points of $\fX$ the paper \cite{christensen2020} contains an error, c.f. \cite[5.0.3]{christensen2020}, as projective modules over the adele ring of $k$ need not 
be trivial. Indeed over a product of rings it is straightforward to construct non-trivial examples of projective modules. 
Only those of locally constant rank are trivial, see Proposition \ref{p:projTriv}. In \S 2 we develop a workaround for a broad class of algebraic stacks that includes 
all algebraic stacks of interest to us. It should be noted however, the essential idea in this workaround is based on the paper \cite{christensen2020}. 
In private email communication Atticus Christensen has indicated that he has a correction to this issue and it is hoped that his arxiv paper will be updated
with this correction in the future. 

 There is a natural map $\fX(k)\to \fX(\bA_k)$ and one says that strong approximation holds for $\fX$ if this map has dense image. Unfortunately, this map 
 rarely has dense image and one is led to consider refinements to strong approximation.

The Brauer group $\Br(\fX)$ produces an obstruction to the closure of $\fX(k)$ being all of $\fX(\bA_{k})$. 
In section \ref{s:BM}, we construct the 
Brauer-Manin pairing for algebraic stacks:
$$
\langle -,- \rangle : \fX(\bA_{k})\times \Br(\fX)\to \QQ/\ZZ 
$$
and prove its main properties needed for this paper. 
The pairing vanishes on $k$-rational points of $\fX(k)$ and the vanishing locus is a closed subset of $\fX(\bA_{k})$. In the case of varieties and schemes
this obstruction is well-known, see \cite{sansuc1981}, \cite{demarche2011}, \cite{Harari2008}, \cite{Morishita1996}, \cite{jlct2009} and \cite{borovoi2013}. 
In \ref{d:strongApprox}, we define strong approximation with respect to a subgroup of the Brauer group. 

Our approach to questions of strong approximation for algebraic stacks of a particular type, see section 2, 
is by reducing the problem to quotients via special groups. This idea is well-known see for example \cite{behrend2007} and \cite{edidin01}. 
Given a quotient stack $[X/G]$ where $G$ is a linear algebraic group over $k$ then we can always write it as a quotient stack of the form $[X'/H]$ where 
 $H$ is a special group. Indeed, we can fix a faithful representation $G\hookrightarrow \gl(V)$ or a representation $G\hookrightarrow \sl(V)$ and write 
 $$
 [X\times_{G}\gl(V)/\gl(V)]\cong [X/G]\cong [X\times_{G}\sl(V)/\sl(V)]. 
 $$
 
A linear algebraic group $G$ always admits a faithful representation into $\sl(V)$, see the discussion at the start of \S \ref{ss:strongApprox}. 
After fixing $G\hookrightarrow \sl(V)$  we are able to produce a stream lined approach to strong approximation as the pullback map 
 $$
 \Br([X/G]) \to \Br(X\times_{G}\sl(V))
 $$
 can be shown to be an isomorphism, see Corollary \ref{c:brIsom}. 
This reduces the question of strong approximation to one about  approximation on $X\times_{G}\sl(V)$. When $X=\spec(k)$ and $G$ is a  connected 
linear algebraic group this has been considered in \cite{borovoi2013} and \cite{jlct2009}. In these papers, a kind of equivariant version of strong approximation,
weaker than strong approximation as we have defined it, is proved. This result is sufficient for our purposes. 
These works are a culmination of a long development of ideas, see also \cite{Harari2008} and \cite{borovoiBrauerManinObstructionsHomogeneous1996}. 
We borrow their ideas to settle
strong approximation for classifying stacks $BG$ where $G$ is a  connected linear algebraic group.
Further examples of algebraic stacks to which strong approximation applies can be generated by considering certain quotients of smooth groupic varieties, see 
\cite{cao2018} and the discussion after Theorem \ref{t:roy}. A familiar example of a groupic variety is a toric variety and our results apply to certain quotient stacks of toric 
varieties.

In section 2 of the paper we introduce the topology on the adelic points of a algebraic stack and prove the main properties of this topology. 
The previously pointed out error in \cite{christensen2020} is corrected here. However, many of the other ideas presented in this section are borrowed 
from \emph{op. cit.}

Section 3 discusses the  Brauer group of an 
algebraic stack is defined, in this paper to be the cohomological Brauer group. We briefly discuss how to spread out a Brauer class in this section. 
As this result is mostly a standard extension of a known result for schemes we only sketch the required argument. 

In section 4, we construct the Brauer-Manin pairing on an algebraic stack
and prove its main properties. In particular, it is shown that the locus inside the adelic points of an algebraic stack consisting of points 
orthogonal to the Brauer-Manin pairing form a closed subset that contains the $k$-rational points. This produces an obstruction to strong approximation
for the entire stack. In the final subsection, we recall the definition of strong  approximation with respect to the Brauer-Manin pairing.

\section*{Acknowledgements}
The author was partially supported by NSERC. The author would like to thank Connor Cassady, Roy Joshua and Atticus Christensen for valuable conversations and
email correspondence. 

The author would also like to thank the two anonymous referees who provided numerous valuable suggestions.

\SetTblrInner[longtblr]{rowsep=0pt}
\begin{longtblr}[caption={Notation}]
    {
        colspec = {cl}}
    $[X/G]$ & If $G$ is a group scheme acting on the scheme $X$ then this is \\  
& the corresponding quotient stack. \\
   $k$& A number field. \\ 
   $\sO_{k}$& The ring of integers in $k$.\\ 
   $\Omega_{k}$ & The set of places of $k$. \\
   $\Omega_{k}^{\fin}$& The set of non-archimedean places of $k$. \\ 
   $\Omega_{k}^{\infty}$ & The set of archimedean places of $k$. \\ 
   $k_{v}$ & The completion of $k$ at the valuation $v$. \\
   $\sO_{v}$ & The valuation subring of $k_{v}$. \\ 
   $k_{S}$ & For some finite subset $S$ of $\Omega_{k}$ this is the ring $\prod_{v\in S} k_{v}$. \\
   $k^{T}$ & For some finite subset $T$ of $\Omega_{k}$ this is the ring $\prod_{v\in T} k_{v}\times \prod_{v\not\in T}\sO_{v}$\\
   $(k^{T})_{\cS}$ & For some finite subsets $S$ and $T$ of $\Omega_{k}$ this is the projection of $k^{T}$ into \\ 
      &   $\prod_{v\not\in S} k_{v}$ \\
   $\bA_{k}$ & The adele ring of $k$ \\ 
   $\bA_{\cS, k}$ & For a finite subset $S\subseteq \Omega_{{k}}$, this is the image of $\bA_{k}$ \\ &under projection into the product $\prod_{v\not\in S}k_{v}$. \\ 
   $\Br(\fX)$ & The Brauer group of a stack or scheme. \\ 
   $\Br_{1}(\fX)$ & The subgroup of the Brauer group trivialised by passage to  an  \\ & algebraic closure. Precisely $\ker(\Br(\fX)\to \Br(\fX\otimes \overline{k} ))$. \\
   $\Br_{a}(\fX)$ & The quotient of the $\Br_{1}(\fX)$ by the image of $\Br(k)$. \\ 
   $G$ & A linear algebraic group over $k$. Often it will be  \\ & assumed to be  connected. \\ 

\end{longtblr}


\section{The topological structure of the adelic points of an algebraic stack} \label{s:topology}

Given an algebraic stack $X$ over a global field, the paper \cite{christensen2020} associates to it a topological 
space with underlying set the adelic points of $X$. Unfortunately, the construction has an error at a critical point, see \cite[Remark 12.0.6]{christensen2020}. 
Indeed given a product of rings then there are non-free finite rank projective modules over such a product so such a product cannot be sufficiently disconnected 
in the terminology of \cite[Definition 5.0.3]{christensen2020}. 

In this paper we will sidestep this issue, and combine the basic idea of \cite{christensen2020} with an  idea from \cite{behrend2007} to construct a topological space associated to 
a suitable category of algebraic stacks over a number field. This category will contain all stacks that we are interested in for this paper.

\subsection{A little lemma}

We will make extensive use of the following. 

\begin{lemma}\label{l:little}
    Let $X/k$ be a scheme of finite type and $H\hookrightarrow G$ an closed inclusion of linear algebraic groups over $k$. Suppose that $H$ acts on $X$. Then 
    we have an isomorphism
    $$
    [X/H] \cong [X\times_{H}G/ G]
    $$
    of quotient stacks. 
\end{lemma}

\begin{proof}
    The proof follows easily from definitions. 
\end{proof} 

\subsection{Vector bundles over the adeles}

\begin{proposition}\label{p:projTriv}
    Let $V$ be a $k$-vector space of finite dimension and let $S\subset \Omega_{k}$ be some finite subset. 
    Every $\gl(V)$-torsor over $\spec(\bA_{\cS,k})$ is trivial. 
\end{proposition}

\begin{proof}
    Giving a principal $\gl(V)$-bundle over $\spec(\bA_{\cS,k})$, for some
finite subset $S\subseteq \Omega_{k}$, is equivalent to giving a finitely generated projective module $P$ over 
$\bA_{\cS,k}$ of constant rank $\dim V$. 
We can write $P\oplus Q\cong \bA_{\cS,k}^{n}$ for some projective module $Q$. It follows that $P$ is the kernel of $$
p: \bA_{\cS,k}^{n}\twoheadrightarrow Q \hookrightarrow \bA_{\cS,k}^{n}
$$
with $p^2=p$. We write $[p]$ for the matrix representing $p$. It is defined over $(k^{T})_{\cS}$ for some finite subset $T\subset\Omega_{k}$ so that 
$$
[p]\in {\rm Mat}_{n\times n}(k^{T}_{\cS}).
$$
For each $v\in\Omega_{k}\setminus S$ we obtain a matrix $$
[p]_{v}\in {\rm Mat}_{n\times n }(k_{v})
$$
which will have coefficients in $\sO_{v}$ whenever $v\not\in T$. Now $\ker([p]_{v})$ is a projective module as $p$ is idempotent, and it has 
rank $\dim V$ as $P$ has constant rank. As $\sO_{v}$ is local we can find isomorphisms 
$$
\ker([p]_{v}) \cong 
\begin{cases}
    k_{v}^{\dim V} & \text{if }v\in T \\ 
    \sO_{v}^{\dim V} & \text{otherwise}. 
\end{cases}
$$
It follows that $\ker([p])$ is a trivial projective $(k^{T})_{\cS}$-module. As $P$ is obtained from it via base extension, it is trivial also. 
\end{proof}

\begin{corollary}
    Every $\sl(V)$-torsor over $\Spec(\bA_{\cS,k})$ is trivial. 
\end{corollary}

\begin{proof}
    A $\sl(V)$-torsor is a $\gl(V)$-torsor equipped with a reduction of structure group to $\sl(V)$. The reduction of structure 
    group amounts to a trivialization of the top exterior power of the projective $\bA_{\cS,k}$-module  associated to the $\gl(V)$-torsor. We 
    have seen that the projective module is trivial so the reduction amounts to a choice of 
    $$
    w\in \bigwedge^{n} \bA_{\cS,k}^{n}
    $$
    that induces an isomorphism 
    $$
    \bA_{\cS,k} \stackrel{\sim}{\longrightarrow}\bigwedge^{n} \bA_{\cS,k}^{n}. 
    $$
    Any two choices of $w$ can be identified via an automorphism of $\bA_{\cS,k}^n$ and hence the associated $\sl(V)$-torsor is trivial. 
    To make this explicit, one can choose a basis, say $e_{i}$, for the free module $\bA_{\cS,k}^{n}$. Then 
    $$
    w= \lambda e_{1}\wedge \ldots \wedge e_{n}
    $$
    where $\lambda\in \bA_{\sC,k}^{\times}$ is an idele. Then one considers the new basis 
    $$
    \lambda e_{1}, e_{2}, \ldots , e_{n}
    $$
    and the associated change of basis automorphism of $\bA_{\cS,k}^{n}$. 
\end{proof}

  
\subsection{Topologies on schemes and algebraic spaces over topological rings}
Given a finite type scheme $X$ over a local field $k_{v}$, the collection of $k_{v}$-points of $X$ has the structure of 
a topological space. For a finite type scheme $X$ over $k$, its adelic points, $X(\bA_{k})$ also has the structure of a topological space.  

Let $R$ be a Hausdorff topological ring.  We denote by $\aff_{R}^{\fint}$ the category of finite type affine schemes over $R$
and let $\top $ be the category of topological spaces. Then there is a functor $$
    F:\aff_{R}^{\fint}\longrightarrow\top
    $$
that is essentially unique. We refer the reader to \cite[Proposition 2.1]{conrad2020} for a construction of this functor, its basic 
properties and the precise meaning of ``essentially unique''.

\begin{remark}\label{r:units}
    Using this results,
    the collection of units $R^{\times}$ obtains a topology from the closed embedding $\Gm \hookrightarrow \AA^{2}$ as the subscheme $V(xy-1)$. 
    There are two potential topologies on $R^{\times}$, one from the proposition and the other from the subspace topology coming from the topology on $R$. 
    These two topologies need not be the same. In fact, for the main example that we will be interested in, the ring of adeles $R=\bA_{k}$ they are different.
\end{remark}

\begin{definition}
    \label{d:contInversion}
    Let $R$ be a Hausdorff topological ring. We say that $R$ has \emph{continuous inversion} if the two topologies on $R^{\times}$ are the same. 
\end{definition}

The fields $k_{v}$ have continuous inversion. 

As a consequence the inverse map on $R^{\times}$ will be continuous with respect to the subspace topology for a ring with continuous inversion. 
When our ring is local and has continuous inversion the functor above extends to functor on finite type schemes, see \cite[Proposition 3.1]{conrad2020}. 

When $R$ is a completion of $k$ or a completion of $\sO_{k}$ we can say slightly more. 

\begin{proposition}\label{p:smoothQuotient}
    Let $f:X\to Y$ be a smooth morphism between finite type $\sO_{v}$-schemes.
    Let $R=k_{v}$ or $\sO_{v}$. 
     Then 
    $$
    f: X(R)\to Y(R)
    $$
    is an open morphism of topological spaces. Hence, the subspace topology on $f(X(R))$ is quotient topology induced from the topology on $X(R)$. 
\end{proposition}

\begin{proof} 
    Being an open map is local on the domain. In view of the previous proposition we may replace $X$ by an open subscheme. 
    After shrinking $Y$, we can find an open affine cover $V_{i}$ of $X$ and local factorizations of $f$ as 
    $$
    f|_{V^{i}} : V_{i} \to \AA^{n}\times Y \stackrel{\rm proj}{\to} Y
    $$
    where the first map is etale.
    From \cite[Lemma 5.3]{conrad2020} the first map induces a local homeomorphism on $k_v$-points and hence is an open map as $X(\sO_v)$ is an open subset of $X(k_v)$. 
     As the topology on the product, is the product
    topology, the projection is also an open morphism. 
\end{proof}

The above observation is the starting point in \cite{christensen2020} in order to topologize points of stacks.

To define a topology on the adelic points of a scheme we need a different approach as the adele ring does not have continuous inversion. Fix a finite subset $S$ of 
$\Omega_{k}$ containing all infinite places. We denote by $\sO_{k,S}$ the intersection of $k$ with $k^{S}$. Alternatively, 
$$
\sO_{k,S} = \{ x\in k \mid v(x)\ge 0\ \forall v\not\in S\}.
$$

There is an inclusion 
$$
\sO_{k,S}\hookrightarrow k^{S}.
$$

\begin{theorem}
    \label{t:adeleTop} Let $X\in \sch^{\fint}_{\sO_{k,S}}$. Then the natural map 
    $$
    X(k^{S})\to \prod_{v\in S}X(k_{v})\times \prod_{v\not\in S}X(\sO_{v})
    $$
    is a bijection. If we use this to equip $X(k^{S})$ with a topology by using the product topology on the right hand side and the above discussion
    then one obtains a functor to topological spaces. 
    The ring $k^S$ is a Hausdorff topological ring so when $X$ is affine the discussion at the start of this section applies to $X(R)$. The two 
    constructed topologies on $X(R)$ agree. 
\end{theorem}

\begin{proof}
    See \cite[3.6]{conrad2020}. 
\end{proof}

For a scheme $X$ of finite type over $k$ one obtains a topological space structure on its adelic points by first spreading $X$ out to a finite type 
scheme over $\sO_{k,S}$ for some subset $S\subset \Omega_{k}$ and then taking a direct limit of the topological spaces $X(k^{S'})$ for $S'\supseteq S$. 
We refer the reader to \cite[Theorem 3.6]{conrad2020} and the surrounding discussion for a complete argument. Spreading out schemes is discussed in 
\cite[3.2]{poonen2017}. 

The analogue of Proposition \ref{p:smoothQuotient} also holds true for adelic points.

\begin{proposition}\label{p:smoothQuotientAdele}
    Let $f:X\to Y$ be a smooth morphism between finite type $\sO_{k,S}$-schemes. 
    Let $R=k^{S}$. Suppose further that the induced map 
    $$
    f: X(\sO_{v})\to Y(\sO_{v})
    $$
    is surjective for almost all $v\not\in S$. 
     Then 
    $$
    f: X(R)\to Y(R)
    $$
    is an open morphism of topological spaces. Hence, the subspace topology on $f(X(R))$ is quotient topology induced from the topology on $X(R)$. 
\end{proposition}

\begin{proof}
    Compare \cite[4.0.5]{christensen2020}. A basic open set of $X(R)$ is of the form, 
    $$
    \prod_{s\in S}U_{s} \times \prod_{t\in T}V_{t} \times \prod_{t\not\in T} X(\sO_{t})
    $$
    where $T\subseteq\Omega_{k}$ is a finite set disjoint from $S$, $V_{t}\subseteq X(\sO_{t})$ is open, 
    and $U_{s}\subseteq X(k_{s})$ is open. The result follows from Proposition \ref{p:smoothQuotient} and the definition of the product topology. 
\end{proof}

We will follow the conventions of \cite[\href{https://stacks.math.columbia.edu/tag/0ELT}{Tag 0ELT}]{stacks-project} regarding algebraic spaces. 
An \emph{algebraic space} $X$ is a a sheaf on the fppf site of $\sch$ such that $X\to X\times X$ is representable and there is a scheme $U$ and a 
surjective \'{e}tale morphism $f:U\to X$. Such data is equivalent to giving an \etale\ equivalence relation on a scheme. In the present context, the equivalence relation 
is recovered as 
$$
U\times_{X} U \xbigtoto{} U. 
$$
The morphism $f$ is called a presentation for $X$.

If the base topological ring is a complete valued field, denoted $k_{v}$, then the above constructions can be extended to algebraic spaces. 

\begin{theorem}
    Denote by ${\rm Alg. Sp}_{k_v}$ the category of quasi-separated algebraic spaces locally of finite type over $k_v$. Then the above functor  can be 
    extended to a functor 
    $$
    {\rm Alg. Sp}_{k_v}\to \top. 
    $$
    This functor preserves fibred products, open and closed immersions. \'{E}tale morphism are sent to local homeomorphisms. 
\end{theorem}

\begin{proof}
    This is in \cite[\S 5]{conrad2020}. The essential idea is given an algebraic space $X$ and a subset $W\subseteq X(k_{v})$, we will say that $W$ 
    is open if and only if $f^{-1}(W)$ is open for every \'{e}tale morphism $f: U\to X$ where $U$ is a finite type scheme. We have previously 
    topologized $U(k)$. 
\end{proof}

\begin{proposition}
    Let $f:X\to Y$ be a smooth morphism between finite type $\sO_{v}$-algebraic spaces. 
    Let $R=k_{v}$ or $\sO_{v}$. 
     Then 
    $$
    f:X(R)\to Y(R)
    $$
    is an open map. Hence the subspace topology on $f(X(R))$ is quotient topology induced from the topology on $X(R)$. 
\end{proposition}

\begin{proof}
   We will implicitly make use of the fact that finite type algebraic spaces are quasi-separated. 
   The proof follows in the same way as for schemes, see Proposition \ref{p:smoothQuotient}. Note that for an \etale\ morphism between 
   algebraic spaces, the assertion is proved in \cite[5.4]{conrad2020}. 
   One can prove the
   required factorisation for smooth maps of algebraic spaces by constructing the factorisation on an atlas. 
\end{proof}

Given a finite type algebraic space over $k$, one can spread it out. We will briefly describe this construction in 
\S \ref{s:BM}.  
One can now construct a topological space structure on the adelic points of a separated algebraic space, mimicking the construction
for schemes. This is carried out in \cite[pg. 90]{conrad2020} which we refer the reader to for further details. Let us just recall a few key components of 
the construction. If $X$ is a separated algebraic space over $\sO_{k,S}$ of finite type then there is a natural bijection 
$$
X(k^{S}) \longrightarrow \prod_{v\in S} X(k_{v})\times \prod_{v\not\in S} X(\sO_{v}). 
$$
The left hand term of the above bijection has the product topology and hence $X(k^S)$ acquires a 
topology via this bijection. 
It follows that if $S\subseteq T$ then the natural map
$$
X(k^{S})\hookrightarrow X(k^{T})
$$
is an inclusion of an open set. Then the topology on
$$
X(\bA_{k})=\colim_{S\subseteq\Omega_k} X(k_{S})
$$
is the colimit topology. 

The smooth quotient property holds also for algebraic spaces. 

\begin{proposition}\label{p:smoothQuotientSpace}
    Let $f:X\to Y$ be a smooth morphism between separated, finite type $\sO_{k,S}$-algebraic spaces.  
    Suppose further that the induced map 
    $$
    f: X(\sO_{v})\to Y(\sO_{v})
    $$
    is surjective for almost all $v\not\in S$. 
     Then 
    $$
    f: X(k^{S})\to Y(k^{S})
    $$
    is an open morphism of topological spaces. Hence, the subspace topology on $f(X(k^{S}))$ is quotient topology induced from the topology on $Y(k^{S})$. 
\end{proposition}

Note that we will briefly discuss how to spread out finite type algebraic spaces in \S \ref{s:BM} so that the last part of the proposition
makes sense. 

\begin{proof}
    This is carried out in \cite[\S 8]{christensen2020} but it readily follows from the discussion above. 
\end{proof}

 
\subsection{Adelic topologies on algebraic stacks}\label{ss:stackTop}

In this subsection we wish to define a topological space structure on $\fX(\bA_{k})$ for a certain class of algebraic stacks $\fX$ over $k$. 
We will discuss how to spread out finite type algebraic stacks in \ref{ss:spreadStack}. 

\begin{definition}
    Let $\fX$ be an algebraic stack of finite type over $k_{v}$. We say that $\fX$ is $k_{v}$-liftable, or just liftable when the context is clear if 
    there is a presentation $P\to \fX$ where $P$ is a finite type algebraic space over $k_{v}$ such that the induced map 
    $$
    P(k_{v})\to \fX(k_{v})
    $$
    is surjective. In this situation, we call $P$ a lifting presentation. 
\end{definition}

\begin{remark}
    Note that we can always find a presentation of $\fX$ in which every $k_v$-point lifts, see \cite[Theorem A.1]{neeraj} and 
    \cite[Definition 1.1]{neeraj}. 
\end{remark}

\begin{definition} \label{d:lifting} Let $S\subset \Omega_{k}$ be a finite set containing all infinite places. 
    Let $\fX$ be an algebraic stack over $\sO_{k,S}$. We say that $\fX$ is \emph{ $S$-liftable} 
    if the diagonal morphism $\fX\to \fX\times_{\sO_{k,S}} \fX$ is separated
    and  
    there is a single presentation $P\to \fX$ with $P$ a separated, finite type $\sO_{k,S}$-algebraic space such that 
    \begin{enumerate}
        \item every $s\in \fX(k^{T})$ lifts
    to $P$ where $T$ is an arbitrary finite subset of $\Omega_{k}$ containing $S$, 
    \item and the induced map $P(\sO_{v})\to \fX(\sO_{v})$ is surjective for all but finitely many $v\not\in S$. 
    \end{enumerate}
    
    In this situation, we will call $P\to \fX$ an $S$-lifting presentation. 

\end{definition}

Examples of $S$-liftable stacks abound. Indeed any quotient stack $X=[P/G]$ where $P$ is a separated, finite type $\sO_{k,S}$-scheme and $G$ is a linear algebraic group over 
$\sO_{k,S}$ is $S$-liftable. To see this, choose a faithful representation $G\hookrightarrow \gl_n$
and 
observe that $[P/G]=[P\times_{G}\gl_{n}/\gl_{n}]$, see Lemma \ref{l:little}. 
To see that the algebraic space $P\times_G \gl_n$ is separated, one considers the morphism 
$$
\lambda: P\times_G \gl_n \to \gl_n/G. 
$$
It suffices to show that this morphism is separated. But this condition is local in the fpqc-topology, 
see \cite[\href{https://stacks.math.columbia.edu/tag/0421}{Tag 0421}]{stacks-project}. 
One now verifies that $\lambda$ is separated by pulling back along the flat morphism $\gl_n\to \gl_n/G$ to obtain the cartesian square
\begin{center}
    \begin{tikzcd}
        P\times \gl_n \ar[r] \ar[d] & \gl_n \ar[d] \\ 
        P\times_G \gl_n \ar[r] & \gl_n/G. 
    \end{tikzcd}
\end{center}
The presentation $P\times_{G}\gl_{n}\to X$ satisfies the first condition of the proposition as
every $\gl_{n}$-torsor over $k^{S}$ is trivial,
see Proposition \ref{p:projTriv} and its proof. Further, 
as the ring $\sO_{v}$ is local, the second condition also holds. In a similar way, examples of $k_{v}$-liftable stacks exist.

\begin{remark}
    Let $\sO_v/m_v=\bar{k_v}$ be the residue field of $\sO_v$ and consider a stack $\fX$ over $\sO_v$ with a presentation $f:P\to \fX$. 
    Suppose that $\fX$ satisfies one of the following conditions. 
    \begin{enumerate}
        \item $\fX$ has affine stabilizers,
        \item $\fX$ has quasi-affine diagonal,
        \item $\fX$ is Deligne-Mumford
    \end{enumerate}
    then to check the second condition of the previous definition it suffices to show that $f$ is surjective on $\bar{k_v}$-points. 
    This follows from \cite[Corollary 1.5]{hall} and formal smoothness. 
\end{remark}

The following elementary lemma gives another method to produce such presentations. 
\begin{lemma}
    Let $\fX\to \fY$ be a separated, representable morphism of algebraic stacks over $\sO_{k,S}$. If $P\to \fY$ is an $S$-liftable presentation then so is 
    $\fX\times_{\fY} P \to \fX$. 
\end{lemma}

\begin{proof}
    Elementary from definitions. 
\end{proof}

If $\fX$ is a liftable algebraic stack over $k_{v}$ then we topologize $\fX(k_{v})$ in the following way. First fix a lifting presentation $P\to \fX$ so that 
$P(k_{v})\to \fX(k_{v})$ surjects. Then equip $\fX(k_{v})$ with the quotient topology. Note that if $X\to Y$ is a morphism of schemes of finite type which is 
smooth and surjective on $k_{v}$-points then the topology on $Y(k_{v})$ is necessarily the quotient topology of that on $X(k_{v})$ by Proposition \ref{p:smoothQuotient}. 

Similarly, if $\fX$ is an $S$-liftable stack we equip $\fX(k^{S})$ with the quotient topology via the morphism
$$
P(k^{S})\to \fX(k^{S}). 
$$

\begin{proposition}
    The above definitions do not depend on the choice of lifting presentation.They are functorial for separated, representable morphisms. 
\end{proposition}

\begin{proof}
    The proof for $S$-liftable stacks will be given. The other case for liftable stacks over $k_v$ is similar, and easier. 

    Suppose that we have two lifting presentations $P_i\to \fX$. 
    In view of the cartesian diagram 
    \begin{center}
        \begin{tikzcd}
            P_1\times_{\fX} P_2 \ar[r] \ar[d] & P_1\times_{\sO_{k,S}} P_2 \ar[d] \\ 
            \fX \ar[r] & \fX\times_{\sO_{k,S}} \fX 
        \end{tikzcd}
    \end{center}
    we see that the algebraic space $P_1\times_{\fX} P_2$ is separated over $\sO_{k,S}$ as we have assumed that the diagonal map 
    $\fX\to  \fX \times_{\sO_{k,S}} \fX $ to be separated. 
    
    One now readily verifies the other conditions in (\ref{d:lifting}) so that 
    $P_{1}\times_{\fX}P_{2} \to \fX$ is also a lifting presentation. 
    The morphism 
    $P_{1}\times_{\fX}P_{2}\to P_{i}$ is smooth and surjective and satisfies the hypothesis of Proposition \ref{p:smoothQuotientSpace}. Hence the continuous map 
    $$
    (P_{1}\times_{\fX}P_{2})(k^S)\to P_{i}(k^S)
    $$
    is a quotient map of topological spaces. The result follows by composing quotients. 

    Consider a separated representable morphism $\fX \to \fY$ of liftable stacks. Then given a lifting presentation $Y\to \fY$ the fibered product $Y\times_{\fY}\fX\to \fX$ is 
    a lifting presentation. The result follows from the independence of the topology from the presentation. 
\end{proof}

\begin{lemma}
    \label{l:isOpen}
    Let $\fX$ be a  $S$-liftable stack over $\sO_{k,S}$. Suppose that $P\to \fX$ is a lifting presentation. 
    Then for every finite subset $T\subset \Omega_{k}$ containing $S$, the stack is $\fX\otimes_{\sO_{k,S}} \sO_{k,T}$ is $T$-liftable. 
    We also have that the natural map 
    $$
    \fX(k^{S})\to  \fX(k^{T})
    $$
    is an open morphism. 
\end{lemma}

\begin{proof}
    The first assertion is clear. 
    The analogous result holds for algebraic spaces, see \cite[\S 5]{conrad2020} so that 
    $$
    P(k^{S})\subseteq P(k^{T})
    $$
    is open. The result now follows from the fact that the topology on $\fX(k^{T})$ is the quotient topology inherited from $P(k^{T})$.  
\end{proof}

Given a  $S$-liftable stack $\fX$ we define a topology on its adelic points by equipping 
$$
\fX(\bA_{k}) = \colim_{T\subseteq\Omega_k} \fX(k^{T})
$$
with the colimit topology in topological spaces. The colimit is over all finite subsets of $\Omega_{k}$ containing $S$.

\section{Cohomology of algebraic stacks and their Brauer groups}

In this section we will review some results on the cohomology of algebraic stacks. We will define the Brauer group  
of an algebraic stack to be its cohomological Brauer group. The technical results of this section will be used in the next 
section to construct the Brauer-Manin pairing for algebraic stacks. 
Where possible, we will keep notation and conventions consistent with those in \cite{stacks-project}. 

\subsection{The big  \'{e}tale site}   

If $\fX$ is a an algebraic stack then $(\sch/\fX)_{\et}$
will denote the big \'{e}tale site on $\fX$. This is the category with objects morphisms
$X\to \fX$ where $X$ is a scheme and morphisms are commuting triangles. The covers are \'{e}tale covers of schemes. 
This site is functorial for morphisms of algebraic stacks, see \cite[\href{https://stacks.math.columbia.edu/tag/06NW}{Tag 06NW}]{stacks-project},
unlike other commonly used sites such as the lisse-\'{e}tale site, see \cite[3.3]{olsson2007}, \cite[4.42]{behrend2003} or 
\cite[\href{https://stacks.math.columbia.edu/tag/07BF}{Tag 07BF}]{stacks-project}.

For every scheme $X$ and every morhpism $f:X\to \fX$ there is an inclusion functor $(\sch/X)_{\et}\to (\sch/\fX)_{\et}$ obtained by 
composing with $f$. 

If $F$ is a sheaf on $(\sch/\fX)_{\et}$ then for every morphism $f:X\to \fX$ we obtain a restricted sheaf on $(\sch/X)_{\et}$ that we denote by 
$F|_{X}$ when the morphism $f$ is clear from the context.

\begin{proposition}\label{c:stacksBaseChange}
    Let $f:\fX\to \fY$ be a representable morphism of algebraic stacks. Let $Y\to \fY$ be a morphism from a scheme to $\fY$. Consider 
    the cartesian diagrm 
    \begin{center}
        \begin{tikzcd}
            \fX_Y \ar[r] \ar[d,"f_Y"] & \fX \ar[d,"f"] \\
            Y \ar[r]  & \fY.  
        \end{tikzcd}
    \end{center}
    Then $R^{i}f_* F|_Y=(R^{i}f_{Y})(F|_Y)$.  
\end{proposition}

\begin{proof}
   This is \cite[\href{https://stacks.math.columbia.edu/tag/075H}{Tag 075H}]{stacks-project}. 
\end{proof}

 
\subsection{The simplicial approach}
The reader is referred to \cite{stacks-project} for a discussion of simplicial objects, spaces, sheaves and their cohomology. 
The goal of this section is to explain how a well-known result about spreading about cohomology classes for schemes, see Theorem \ref{t:cohSpread} and Corollary \ref{c:cohGeneric},
can be generalised to algebraic stacks, see Theorem \ref{t:cohSpreadStack}. 

We will make use of the following result from SGA 4. 

\begin{theorem} \label{t:cohSpread}
Let $I$ be a filtered ordered set thought of as a category. Consider a 
$$
X:I^{op}\to \sch
$$
filtered inverse system of schemes. We write $X(i)=X_{i}$. Assume that $X_{i}$ are quasi-compact and quasi-separated and that all morphisms in the system 
are affine. 
\begin{enumerate}
    \item Then the limit $X=\varprojlim X_{i}$ exists as a scheme. 
    \item Fix $0\in I$ and suppose that $F_{0}$ is a sheaf on the \'{e}tale site of $X_{0}$. If $i\ge 0$ we define $F_{i}=u_{i}^{-1}F_{0}$ where $u_{i}:X_{i}\to X_{0}$
    is the morphism defining the inverse system. Also set $F_{\infty}= u^{-1}F_{0}$ where $u:X\to X_{0}$ is the universal morphism. Then 
    $$
    \varinjlim H^{n}(X_{i},F_{i})\cong H^{n}(X,F_{\infty}). 
    $$
\end{enumerate} 
\end{theorem}

\begin{proof}
See \cite[Expose V, 5.7,5.8]{SGA4}. 
\end{proof}

We will often apply this with $F_{0}=\Gm$. In this case we have $F_{i}=u_{i}^{-1}\Gm =\Gm$ as we are using the big \etale\ site, cf. 
\cite[\href{https://stacks.math.columbia.edu/tag/04DI}{Tag 04DI}]{stacks-project}.

\begin{corollary}\label{c:cohGeneric}
    Let $R$ be an integral domain with field of fractions $K$. Let $\tilde{X}$ be a scheme of finite type over $R$ and $F$ a sheaf on $\tilde{X}$. 
    Write $X=\tilde{X}_{K}$ and for each $a\in R$ set $X_{a}=\tilde{X}\otimes_{R}R_{a}$. 
    If $\sigma \in H^{n}(X,F|_{X})$ then there is an $a\in R$ and a $\tau\in H^{n}(X_{a},F)$ so that $\tau|_{K}=\sigma$. 
\end{corollary}

\begin{proof}
    We have $X_{K}= \varprojlim_{a\in R} X_{a}$. So the result follows from the theorem. For an important special case see also \cite[6.4.3]{poonen2017}. 
\end{proof}

If $X_{\bullet}$ is a simplicial scheme (or algebraic space) and $F$ a simplicial sheaf for the \'{e}tale site on $X_{\bullet}$ then we have a first quadrant spectral sequence
with 
$$
E_{1}^{pq}(X_{\bullet}, F)=H^{q}(X_{p},F_{p}) \implies H^{p+q}(X_{\bullet}, F),
$$
see \cite[\href{https://stacks.math.columbia.edu/tag/0D76}{Tag 0D76}]{stacks-project}. 
The spectral sequence is functorial in that given a morphism $f:Y_{\bullet}\to X_{\bullet}$ of simplicial schemes (or algebraic spaces) we obtain a morphism of spectral sequences 
$$
f^{*}:E_{1}^{pq}(X_{\bullet}, F)\to E_{1}^{pq}(Y_{\bullet},f^{-1}F). 
$$

In this situation we will say that a class $\sigma\in E_{1}^{pq}(Y_{\bullet}, f^{-1}F)$
lifts to $X_{\bullet}$ if there is a $\tau\in E_{1}^{pq}(X_{\bullet},F)$ with $f^{*}(\tau)=\sigma$. Alternatively, we say that the class $\tau$ lifts $\sigma$. 

\begin{lemma}\label{l:ssSpread}
Let $R$ be an integral domain with field of fractions $K$. Let $\tX_{\bullet}$ be a simplicial scheme over $R$ so that each $\tX_{n}$ is of finite type over 
 $R$. Let $F$ be a sheaf on $\tX_{\bullet}$. For $a\in R $ we write $\tX_{\bullet,a}$ (resp. $\tX_{\bullet,K}$) for $\tX_{\bullet}\otimes_{R}R_{a}$ 
 (resp. $\tX_{\bullet}\otimes_{R}K$) so that we have morphisms of simplicial schemes $\tX_{\bullet,K}\to \tX_{\bullet,a}$. 
  We write $F_{a}$ (resp. $F_{K}$) for the restriction of $F$ to $\tX_{\bullet,a}$ (resp. $X_{K}$). If $\sigma\in E_{r}^{pq}(\tX_{\bullet,K},F_{K})$ then there is 
  an $a\in R$ and $\tau\in E_{r}^{pq}(\tX_{\bullet,a},F_{a})$ that lifts $\sigma$. 
\end{lemma}

\begin{proof}
We argue by induction on $r$. For $r=1$ the result follows from the previous corollary and the fact that 
$$
\tX_{\bullet,K}=\varprojlim_{a\in R} \tX_{\bullet,a}. 
$$

For a general $r$, we have that 
$$
\ker(E_{r-1}^{pq}\to E_{r-1}^{p+r-1,q-r+2})\twoheadrightarrow E_{r}^{pq}
$$
so that $\sigma$ comes from a class in $E_{r-1}^{pq}$ which lifts. 
\end{proof}

\begin{corollary}\label{c:cohomolgyExtendSS}
    In the situation of the theorem, let 
    $$
    \sigma\in H^{n}(\tX_{\bullet,K},F_{K}). 
    $$
    Then there is an $a\in R$ so that $\sigma$ lifts to a class in 
    $$
    H^{n}(\tX_{\bullet,a},F_{a}). 
    $$
\end{corollary}

\begin{proof}
This follows from the fact that the spectral sequence converges to the cohomology of the simplicial sheaf on the given simplicial scheme, 
see \cite[\href{https://stacks.math.columbia.edu/tag/09WJ}{Tag 09WJ}]{stacks-project}. 
\end{proof}

These spectral sequences can be used to extend Corollary \ref{c:cohGeneric} to algebraic spaces. 

\begin{proposition}
    Let $R$ be an integral domain with field of fractions $K$. Let $\tilde{X}$ be an algebraic space of finite type over $R$ and $F$ a sheaf on $\tilde{X}$. 
    Write $X=\tilde{X}_{K}$ and for each $a\in R$ set $X_{a}=\tilde{X}\otimes_{R}R_{a}$. 
    If $\sigma \in H^{n}(X,F|_{X})$ then there is an $a\in R$ and a $\tau\in H^{n}(X_{a},F)$ so that $\tau|_{K}=\sigma$. 
\end{proposition}

\begin{proof}
    Let $U\to X$ be an \etale\ presentation for $X$. We expand it to simplicial scheme $u:U_{\bullet}\to X$ by taking the Cech nerve. Then $u^{-1}F$ is 
    a simplicial sheaf on $U_{\bullet}$. Then  the result follows from \cite[\href{https://stacks.math.columbia.edu/tag/06XJ}{Tag 06XJ}]{stacks-project} and Lemma \ref{l:ssSpread}. 
\end{proof}

Another application of the same idea yields the result for algebraic stacks. 

\begin{theorem}\label{t:cohSpreadStack}
    Let $R$ be an integral domain with field of fractions $K$. 
Let $\ftX$ be an algebraic stack over $R$ of finite type so that it has a presentation $X\to \ftX$ so that $X$ is an $R$-scheme of finite type. 
If $F$ is a sheaf on the \'{e}tale site of $\ftX$ and 
$\sigma \in H^{n}(\ftX_{K},F_{K})$ then there is an $a\in R$ so that $\sigma$ lifts to $\ftX_{a}$. 
\end{theorem}

\begin{proof}
    We start by remarking that, by using the previous proposition, we can now prove an analogue of 
    Lemma \ref{l:ssSpread} for algebraic spaces. 

    Choose a presentation $X\to \ftX$ as in the statement of the theorem. 
    Then the simplicial space $X_{\bullet}$ formed by taking the coskeleton of $u:X\to \ftX$
    has the same cohomology with coefficients in $u^{-1}F$ as the stack. This follows from 
 \cite[\href{https://stacks.math.columbia.edu/tag/06XF}{Tag 06XF}]{stacks-project}.  
    The remark at the start of this prove now completes the proof of the theorem. 
\end{proof}

 
\subsection{The Brauer group of an algebraic stack}

Let $\fX$ be an algebraic stack. We define its Brauer group to be $$ \Br(X):=H^{2}((\sch/\fX)_{\et},\Gm_{\fX})_{tors}.$$
One could potentially use the lisse-etale site to define this but this makes no difference, see \cite[A.1]{heinloth2009}.

If $\fX$ is an algebraic stack over $k$, our number field we introduce some variations on the Brauer group that will play an important role later. 
We define 
$$
\Br_{1}(\fX) := \ker\left( \Br(\fX) \to \Br(\fX\otimes_{k} \overline{k})\right)
$$
and
$$
\Br_{a}(\fX) := \Br_{1}(\fX)/\left( p^{*}\Br(k)\right) 
$$
where $p:\fX\to \Spec(k)$ is the structure map. 

These definitions are just adaptations to algebraic stacks of the corresponding constructions in \cite{sansuc1981}.

\section{The Brauer-Manin pairing for algebraic stacks}\label{s:BM}

 
\subsection{Spreading out schemes, algebraic spaces and their morphisms}

The following result is well-known and we record it here for future use. 

\begin{theorem}\label{t:algspaceSpread}
Let $R$ be an integral domain with field of fractions $K$.
\begin{enumerate}
    \item If $X$ is an algebraic space of finite type over $K$ then there is open subscheme $V\subseteq \spec R$ and an algebraic space $\tilde{X}\to V$ whose generic 
fibre is $X$. Further, $\tilde{X}$ is of finite presentation over $V$. 
\item If $ \tilde{X} $ and $ \tilde{Y} $ are algebraic spaces of finite presentation over $\spec R$ and $f: \tilde{X}_{K} \to \tilde{Y}_{K}$ is a morphism 
over their generic fibers then $f$ can be lifted to a morphism $ \tilde{f}: \tilde{X}_{V} \to \tilde{Y}_{V}$ for some open subscheme $V\subseteq \spec R$.  
\item In the situation of the previous part, if $f$ is smooth or \etale\  then there is an open subscheme of $\spec R$ over which the lift is smooth or \etale. 
\item  In the above situation suppose that we have finitely presented algebraic spaces $ \tilde{X},\ \tilde{Y}$ and $ \tilde{Z}$ over $S$ with generic fibers $X,\ Y$ and $Z$. If 
$ \tilde{f} : \tilde{X} \to \tilde{Y} $,\ $ \tilde{g}: \tilde{Y}\to \tilde{Z}$ and $ \tilde{h}: \tilde{X}\to \tilde{Z}$ are morphisms whose generic fibers satisfy 
$ g\circ f =h$ then there is an open subscheme $U$ of $S$ with
$$
\tilde{g} \circ \tilde{f} = \tilde{h}
$$
over $U$. 
\end{enumerate}
\end{theorem}

\begin{proof} 
    For a scheme \cite[3.2.1]{poonen2017} proves the first three parts. For a scheme the last part is an elementary result of commutative algebra. 
    
    We move on to the case of an algebraic space. 
    
    The fact that $X$ is of finite type means that we can find a presentation $U\to X$ with $U$ of finite type over $K$, 
    see \cite[\href{https://stacks.math.columbia.edu/tag/03XE}{Tag 03XE}]{stacks-project}. It follows that $U\times_{X} U$ is also of finite type. 
By the result for schemes, we can find an open subscheme $V\subseteq \spec R$ and schemes of finite presentation $ \tilde{Q} $ and $ \tilde{U}$ that extend
$U\times_{X}U$ and $U$ respectively. By the same result, we can assume that the two projections $U\times_{X}U \to U$ extend to \etale\ morphisms and 
the morphism $U\times_{X}U\to U\times_{K}U$ extends to a monomorphism 
$$
m: \tilde{Q} \to \tilde{U}\times_{R} \tilde{U}. 
$$
It is claimed that after further restricting $V$ we can assume that this data forms an \etale\ equivalence relation. 
For example, to check the first condition of being an equivalence relation, reflexivity, we need to show that the diagonal map $ \tilde{U}\to \tilde{U} \times_{R} \tilde{U}$
factors through $ \tilde{Q}$. This follows from the result for schemes. The other axioms are proved similarly.

The quotient algebraic space $ \tilde{U}/ \tilde{Q}$, see \cite[\href{https://stacks.math.columbia.edu/tag/02WW}{Tag 02WW}]{stacks-project}, is the required extension of $X$. 

The remaining two parts follow from the known result for schemes by lifting the morphisms on atlases. Note that a morphism of algebraic spaces is smooth or \etale\ 
if the induced morphism on atlases is. 
\end{proof}

 
\subsection{Spreading out algebraic stacks} \label{ss:spreadStack}

\begin{theorem}\label{t:stackSpread}
Let $R$ be an integral domain with function field $K$. Let $\fX$ be an algebraic stack of finite type over $K$. Then there is an algebraic stack 
$\tilde{\fX}$ over an open subscheme $U\subseteq \spec R$ whose generic fiber is $\fX$. 
\end{theorem}

\begin{proof} 
Once again, this is a standard result so we provide a brief sketch of the argument only. 

Consider a presentation $X\to \fX$. In view of the hypothesis on $\fX$ the scheme $X$ can be chosen to be of finite type. As smooth morphisms are of finite presentation we 
have that all the products $X\times_{\fX} \times_{\fX} \dots \times_{\fX}X $ are of finite type over $K$. This data forms a groupoid in algebraic spaces 
$(X,X\times_{\fX}X,\pi_{1},\pi_{2}),\pi_{13})$,
\cite[\href{https://stacks.math.columbia.edu/tag/0231}{Tag 0231}]{stacks-project}, \cite[\href{https://stacks.math.columbia.edu/tag/0437}{Tag 0437}]{stacks-project} and 
\cite[2.4.3]{laumon2000}. We can find an open subscheme $U\subseteq \spec R$ to which all the defining data 
of this groupoid lifts. We require this lift to form a groupoid, which amounts to certain diagrams commuting as in \cite[2.4.3]{laumon2000}. By further refinement, 
 we can find $U'\subseteq U\subseteq \spec R$ over which we have a groupoid in schemes. To this groupoid in schemes there is an associated algebraic stack,
see \cite[4.3.1]{laumon2000}. This algebraic stack lifts $\fX$. 
\end{proof}

\begin{corollary}
    \label{t:spreadBrauer}
We work in the situation of the previous theorem. Let $x\in\Br(\fX)$ be a Brauer class. Then there is an open subscheme $V\subseteq U$
and a Brauer class $ \overline{x} \in \Br( \tilde{\fX} |_V)$ extending $x$. 
\end{corollary}

\begin{proof}
  This follows from Theorem \ref{t:cohSpreadStack}. 
\end{proof}

\subsection{The Brauer-Manin pairing on algebraic stacks}

In this subsection we will construct the Brauer-Manin pairing on an algebraic stack of finite type over a number field. Before doing
so, let's recall a few facts pertaining to the construction of the pairing for schemes. 

For each $v\in\Omega_{k}$, 
local class field theory constructs a morphism
$$
\inv_{k,v}:\Br(k_{v})\to \QQ/\ZZ,
$$
known as the Hasse invariant. The subscripts $k$ and $v$ will frequently be dropped from the notation when the context makes them clear. 
For finite places, this is an isomorphism. Further there is a short exact sequence 
$$
0\to \Br(k)\to \bigoplus_{v\in\Omega_k}\Br(k_{v}) \stackrel{\sum}{\to}\QQ/\ZZ. 
$$
If $X$ is a finite type scheme over $k$ then given an adelic point $x\in X(\bA_{k})$ we obtain for each $v\in\Omega_{k}$
a point $x_{v}\in X(k_{v})$ by restricting along the projection $\bA_{k}\to k_{v}$. The scheme can be spread out to a scheme over 
$\sO_{k,S}$ using Theorem \ref{t:algspaceSpread}. 
All but finitely many of the $x_{v}$ will lie inside
$X(\sO_{v})$ where $v\not\in S$. 

Given a Brauer class $b\in \Br(X)$, we can also spread it out to a Brauer class over $\sO_{k,S}$ by potentially increasing $S$, see Theorem \ref{t:spreadBrauer}. 
As $\Br(\sO_{v})=0$ we obtain a pairing 
$$
X(\bA_{k})\times \Br(X) \to \QQ/\ZZ 
$$
given by 
$$
\langle x,b\rangle = \sum_{v\in\Omega_k} \inv_{k,v}(x_{v}^{*}(b)). 
$$
The sum is finite in view of the remarks above.

Now let $\fX$ be a finite type algebraic stack over $k$. We can assume that it lifts to an algebraic stack over $\sO_{k,S}$. 
Let $\bA_{k}$ be the ring of adeles of $k$. Given an adelic point $x\in\fX(\bA_{k})$ we obtain for each $v\in \Omega_{k}$ a $k_{v}$-point by restricting along 
the projection 
$$
\bA_{k}\to k_{v}. 
$$
We will denote this point by $x_{v}$. All but finitely many of these will lift to $\sO_{v}$-points. 

We define a pairing, the Brauer-Manin pairing, 
$$
\langle,\rangle: \fX(\bA_{k})\times \Br(\fX) \to \QQ/\ZZ 
$$
given by 
$$
\langle(x_{v}),b\rangle = \sum_{v\in S} \inv_{k,v}x_v^* (b).
$$

\begin{proposition}
The pairing constructed above exists, in other words the sum is finite. 
\end{proposition}

\begin{proof}
    By Theorem \ref{t:stackSpread}  there is a $f\in \sO_{k}$ so that the 
    stack $\fX$ lifts to a stack $\ftX$ over  $\sO_{k}[f^{-1}]$. 
    We can further assume that the Brauer class lifts to $\ftX$ by Theorem \ref{t:spreadBrauer}. 
    There are only finitely many primes outside $\Spec(\sO_{k}[f^{-1}])$ and the Brauer class 
    $x_{v}^{*}(b)$ vanishes for each $v\in\Spec(\sO_{k}[f^{-1}])$. The reason is that the Brauer group
    $
    \Br(\sO_{v})=0
    $ is trivial. 
\end{proof}

For a subset $B\subseteq \Br(\fX)$ we denote by $\fX(\bA_{k})^B$ the $B$-fixed point locus (or $B$-orthogonal locus) of this pairing. 
It is the subset 
$$
\fX(\bA_{k})^{B}:=\{ 
    x\in \fX(\bA_{k}) \mid \langle b,x\rangle  =0,\ \forall\, b\in B 
\}. 
$$

 
\subsection{Continuity of the Brauer-Manin pairing}

\begin{proposition}
    Let $k_{v}$ be a local field. 
    Let $\fX$ a finite type liftable  algebraic stack over $k_{v}$. Let $b\in \Br(\fX)$. Then the map 
    $$
    \fX(k_{v})\to \QQ/\ZZ
    $$
    given by 
    $$
    x\mapsto \inv(x^{*}(b))
    $$
    is continuous. 
\end{proposition}

\begin{proof}
    First assume that $\fX$ is an algebraic space. Using \cite[Proposition 5.4]{conrad2020} we can find a finite type scheme $X$
    and \etale\ morphism $f:X\to\fX$ so that a point $ \overline{x} \in \fX(k_{v})$ lifts to a point $x\in X$. The morphism 
    $X\to \fX$ is a local homeomorphism. By \cite[pg. 235]{poonen2017} there is an open neighborhood of $x$ on which 
    the map
    $$
    y\mapsto \inv(y^{*}f^{*}b)
    $$
    is constant. The result for algebraic spaces follows. 

    Now suppose $\fX$ is a finite type liftable algebraic stack. 
    Fix a point $ \overline{x}\in\fX(k_{v})$. 
    There is a presentation by a finite type algebraic space $f:X\to \fX$ such that $ \overline{x} $ lifts to a point $x\in X(k_{v})$. 
    Then we have a diagram
    \begin{center}
        \begin{tikzcd}
            X(k_v)  \ar[dr, "f"] \ar[d, "q"] & \\ 
            \fX(k_v)  \ar[r, "g"] & \QQ/\ZZ
        \end{tikzcd}
    \end{center}
    with $q$ a quotient map of topological spaces. Then the continuity of $g$ follows from the continuity of $f$. 
\end{proof}

\begin{proposition} \label{p:brauerClosed}
    Let $\fX$  be an algebraic stack over a number field $k$. Assume that $\fX$ is $S$-liftable and of finite type. For each $b\in \Br(\fX)$
    the function
    $$
    \fX(\bA_{k})\to \QQ/\ZZ 
    $$
    given by 
    $$
    x\mapsto \langle b,x\rangle
    $$
    is continuous.

\end{proposition}

\begin{proof}
   By Theorem  \ref{t:stackSpread}, we may assume that $\fX$ is defined over some $\sO_{k,T}$. Further we can assume that $b$ is defined over 
   $\sO_{k,T}$ by Theorem \ref{t:spreadBrauer}. 

   Arguing as in the proof of the previous proposition, for each $T\supseteq S$, the function 
    $$
    \fX(k^T)\to \QQ/\ZZ\quad x\mapsto \langle b,x\rangle
    $$
    is continuous. The result follows from the construction of filtered colimits of topological spaces. 
\end{proof}

\begin{definition}\label{d:strongApprox}
    Let $S$ be a finite subset of $\Omega_{k}$ and let $\fX$ be an $S$-liftable algebraic stack over $k$. If $B$ is a subgroup of 
    $\Br(\fX)$ then we say that \emph{strong approximation for $\fX$ holds off $S$ with respect to $B$} if the diagonal image of $\fX(k)$ is dense inside 
    $\fX(\bA_{\cS,k})^B$.
\end{definition}

Our results below will verify this property for certain stacks, in particular for classifying stacks of  connected linear algebraic groups over $k$.

\section{Strong approximation for quotient stacks}

\subsection{Some facts about $\sl(V)$}

We let $V$ be a finite dimensional $k$-vector space and denote by $\sl(V)$ the special linear group of $V$.

\begin{proposition} \label{p:schemeBaseChange} 
    Let $X$ be a smooth $k$-variety. Consider the projection map $X\times_{k}\sl(V)\to X$. Then it induces isomorphisms 
    \begin{enumerate}
        \item $H^{0}(X,\Gm)\cong H^{0}(X\times_{k}\sl(V),\Gm)  $, 
        \item $H^{1}(X,\Gm)\cong H^{1}(X\times_{k}\sl(V),\Gm) $,
        \item $H^{2}(X,\Gm)\cong H^{2}(X\times_{k}\sl(V),\Gm) $. 
    \end{enumerate}
\end{proposition}

\begin{proof}(Thanks to the referee.)
    The first assertion follows from \cite[6.5]{sansuc1981}. Note that in the notation of Sansuc, $U(\sl(V)):= k[\sl(V)]^\times/k^{\times}=0$. 
    This follows from \cite[theorem 3]{rosenlicht1961}, using the fact that $\sl(V)$ is equal to it's derived subgroup.  

    The second assertion follows immediately from \cite[6.6]{sansuc1981} and \cite[Lemma 6.9]{sansuc1981}.

    For the final assertion we start by observing that 
    $$
    H^{2}(X\otimes_{k} \overline{k}, \Gm) \cong H^{2}((X\times \sl(V))\otimes_{k} \overline{k}, \Gm), 
    $$
    via the projection map. To see this one applies the Kummer sequence
    $$
    1\to \mu_{n}\to \Gm \to \Gm \to 1
    $$
    and the fact that the proof of assertion (2) works over an algebraically closed field to reduce the question to showing that 
    $$
    H^{2}(X\otimes_{k} \overline{k}, \mu_{n}) \cong H^{2}((X\times \sl(V))\otimes_{k} \overline{k}, \mu_{n}).  
    $$
    Using the comparison theorem for cohomology, see \cite[Exp. XI]{SGA4}, we have 
    $$
    H^{1}(\sl(V_{ \overline{k}}), \mu_{n})=H^{2}(\sl(V_{ \overline{k}}), \mu_{n})=0
    $$
    by \cite[pg. 148, Thm 6.5]{lie} and the universal coefficient theorems. 
    The equality of $H^{2}$ with $\mu_{n}$-coefficients results from the Kunneth formula and hence we have the required isomorphism 
    of Brauer groups over an algebraically closed field. 

    To pass to a non-algebraically closed field we consider $\Br_1(Y) := \ker\left( \Br(Y)\to \Br(Y\otimes_k \bar{k})\right)$ and 
    and $\Br_a(Y) := \Br_1(Y)/\Br(k)$ where $Y$ is some algebraic variety. 

    As the the projection map $X\times \sl(V)\to X$ has a section we just need to show that the map 
    $$
    \Br(X)\to \Br(X\times_k\sl(V))
    $$
    is surjective.  
    By applications of the five lemma we are reduced to showing that the map induced by projection
    $$
    \Br_a(X)\to \Br_a(X\times_k \sl(V))
    $$
    is surjective. But this follows from \cite[6.6 (ii)]{sansuc1981} combined with \cite[Lemma 6.9 (iii)]{sansuc1981}.
\end{proof}

\begin{corollary}
    In the above situation, consider the projection map $f:X\times \sl(V)\to X$. Then we have 
    \begin{enumerate}
        \item $f_{*}\Gm \cong \Gm$,
        \item $R^{1}f_{*}\Gm =0$,
        \item $R^{2}f_{*}\Gm=0$. 
    \end{enumerate}
\end{corollary}

\begin{proof}
    The etale site has enough points so the statements can be checked on stalks. Let $x:\spec( \overline{k} )\to X$ be a geometric point. Denote by
    $\sO^{sh}_{X,x}$ the strict Henselisation of the local ring at the image of $x$. We have a diagram 
    \begin{center}
        \begin{tikzcd}
            \Spec(\sO^{sh}_{X,x})\times \sl(V) \ar[d]\ar["p",r] & X\times\sl(V) \ar[d] \\
            \Spec(\sO^{sh}_{X,x})\ar[r] & X. 
        \end{tikzcd}
    \end{center}
    We have $p^{-1}\Gm =\Gm$, \cite[\href{https://stacks.math.columbia.edu/tag/04DI}{Tag 04DI}]{stacks-project}. The stalks can be computed as 
    $$
    H^{i}(\Spec(\sO^{sh}_{X,x})\times \sl(V),\Gm), 
    $$
    see \cite[\href{https://stacks.math.columbia.edu/tag/03Q7}{Tag 03Q7}]{stacks-project}. 
    Furthermore, let $N_{x}$ be the category of affine \etale\ neighbourhood of $x$. So that 
    $$
    \Spec(\sO^{sh}_{X,x})\times \sl(V) = \lim_{{U\in}N_{x}} N\times \sl(V), 
    $$
    so we may apply \cite[Expose V, 5.7-5.8]{SGA4}.  In view of the isomorphisms 
    $$
    H^{i}(U\times \sl(V),\Gm)\cong H^{i}(U,\Gm)
    $$
    we obtain 
    $$
    H^{i}(\Spec(\sO^{sh}_{X,x})\times \sl(V),\Gm)\cong H^{i}(\Spec(\sO^{sh}_{X,x}),\Gm)=0\quad i=1,2
    $$
    using the fact that $\sO^{sh}_{X,x}$ is local and \cite[pg. 148,2.13]{milne1980}. 
    The first assertion is proved in a similar way using the definition of the pushforward of sheaf on the \etale\ topos. 
\end{proof}

\subsection{Strong Approximation}\label{ss:strongApprox}

We will equip Artin stacks with their big \'{e}tale site. This has the advantage over the lisse-etale site in that it is functorial. 
Recall that we have defined the Brauer group of a stack to be the cohomological Brauer group in this topology. 

In this section $G/k$ will be a   linear algebraic group. Fix a faithful representation $G\hookrightarrow \sl(V)$. 
Such a representation always exists. For example, in view of the fact that $G$ is linear we can find a representation 
$G\hookrightarrow \gl_{n}$. Then we may compose it with the representation $\gl_{n}\hookrightarrow \sl_{2n}$ given by 
$$
M \longmapsto \begin{pmatrix}
    M & 0 \\
    0 & (M^t)^{-1}
\end{pmatrix}. 
$$

Given a quotient stack of the form $[Y/G]$ we may write it as an $\sl(V)$-quotient
$$
[Y/G]\cong [Y\times_{G}\sl(V)/\sl(V)],
$$
see Lemma \ref{l:little}. 
The object $X:= Y\times_{G}\sl(V)$ always exists as an algebraic space. If $G$ is reductive and $Y$ is quasi-projective with a linearised action then
the quotient will be a scheme, see \cite{mumford}.

\begin{proposition} 
    \label{c:brIsom}
    Let $X$ be a smooth variety with an action of $\sl(V)$. 
    Consider the morphism $$f:X\to [X/\sl(V)].$$ Then
    $$
    f^{*}: \Br([X/\sl(V)]) \to \Br(X)
    $$
    is an isomorphism. 
\end{proposition}

\begin{proof} 
   We have a cartesian diagram 
   \begin{center}
    \begin{tikzcd}
        X\times\sl(V) \ar[r]\ar[d] & X\ar[d] \\
        X\ar[r] & \left[X/\sl(V) \right]. 
    \end{tikzcd}
   \end{center}
   The top horizontal map is the action and the left vertical map is projection. There is an automorphism of $X\times\sl(V)$ that switches these, so both maps are 
   smooth. 
   We may apply the spectral sequence \cite[\href{https://stacks.math.columbia.edu/tag/06XJ}{Tag 06XJ}]{stacks-project}
   to both of these maps. 
   The spectral sequence for $X\to [X/\sl(V)]$ has  
   $$
   E_{1}^{pq} = H^{q}(X\times \sl(V)^{\times p}, \Gm). 
   $$
   The spectral sequence for the action $X\times \sl(V)\to X$ has 
   $$
   E_{1}^{pq} = H^{q}(X\times \sl(V)^{\times p+1}, \Gm). 
   $$
   The induced maps 
   $$
   H^{q}(X\times \sl(V)^{\times p}, \Gm) \to H^{q}(X\times \sl(V)^{\times p+1}, \Gm). 
   $$
   are induced by projections so the result follows from Proposition \ref{p:schemeBaseChange}. 
\end{proof}

\begin{theorem}\label{t:roy}
    Let $S$ be a finite subset of $\Omega_{k}$ containing all infinite places. 
    Let $G$ be a linear algebraic group defined and smooth over $\sO_{k,S}$. Fix a faithful representation 
    $$
    G\hookrightarrow \sl(V)
    $$
    over $\sO_{k,S}$. Suppose that $Y$ is a smooth and separated algebraic space over $\sO_{k,S}$ with an action of $G$. 
    If $(Y\times_G \sl(V))\otimes_{\sO_{k,S}} k$ is a variety and strong approximation holds for 
    $Y\times_{G}\sl(V)$ off $S$ with respect to $\Br(Y\times_{G}\sl(V))$. Then strong 
    approximation holds for $[Y/G]$ off $S$ with respect to $\Br([Y\times_{G}\sl(V)/\sl(V)])$. 
\end{theorem}

\begin{proof}
    In what follows we will write $f$ for the presentation 
    $$
    Y\times_{G}\sl(V) \to [Y\times_{G}\sl(V)/\sl(V)]. 
    $$
    We will abuse notation and write $f$ for the map induced on $R$-points for various rings $R$. 

    The discussion after \ref{d:lifting} shows that the presentation $f$ is a lifting presentation. 
    Further, the same argument in that discussion shows that
    $$
    Y\times_G \sl(V) \to \sl(V)/G
    $$
    is a smooth morphism as being smooth is local in the smooth topology, see \cite[\href{https://stacks.math.columbia.edu/tag/06FC}{Tag 06FC}]{stacks-project}. 
    It follows that $Y\times_G \sl(V)$ is smooth over $\sO_{k,S}$. 

    Let $x\in [Y\times_{G}\sl(V)/\sl(V)](\bA_{\cS,k})$  be a point orthogonal to $\Br([Y\times_{G}\sl(V)/\sl(V)])$. We can lift it to a point 
    $y\in Y\times_{G}\sl(V)(\bA_{\cS,k})$. In view of the above proposition, if $x$ is orthogonal to $\Br([Y\times_{G}\sl(V)/\sl(V)])$ then 
    $y$ is orthogonal to $\Br(Y\times_{G}\sl(V))$ as the Brauer-Manin pairing is easily seen to be functorial, i.e 
    $$
    \langle x,b\rangle = \langle y,f_{*}b\rangle. 
    $$
    If $U$ is an open neighbourhood of $x$ in the adelic topology then $f^{-1}(U)$ is an open neighbourhood of $y$ in the 
    adelic topology. Then there a $k$-point $t\in Y\times_{G}\sl(V)(k)$ in $f^{-1}(U)$. It follows that $f(t)$ is a $k$-point in $U$. 
\end{proof}

To generate examples to which this theorem applies, we may use the ideas of \cite{cao2018}. A variety with $X$ with $G$-action is said to be groupic
if there is a dense equivariant open subset of $X$ that is isomorphic to $G$ as a $G$-variety. It is easy to see that if $X$ is $G$-groupic 
then $X\times_{G}\sl(V)$ is $\sl(V)$ -groupic as 
$$
G\times_{G}\sl(V)\cong\sl(V). 
$$
The main theorem of \cite[1.3]{cao2018} can be used to generate examples. In this vein, quotients of toric varieties by included tori can be used to provide examples.

We change direction now and consider classifying stacks, which are not covered by the above examples. 
In \cite[6.1]{borovoi2013} and \cite[3.7(b)]{jlct2009} a kind of equivariant strong approximation theorem for homogeneous spaces is proved. More precisely consider $H$ a simply connected semi-simple 
linear algebraic group and $G$ a connected subgroup. Let $X=G\backslash H$. If $S$ is a finite set of places containing 
such that strong approximation holds for $H$ off $S$. 
Then \cite[6.1]{borovoi2013} says that $X(k)H(k_S)$ is dense 
is dense in the Brauer-Manin fixed point locus of $X(\bA_{\cS,k})$. We will make use of this result with $H=\sl(V)$. The fact that strong approximation holds for $\sl(V)$ 
(off S) follows from a classical theorem of Knesser. 

\begin{theorem}\label{t:sApproxBG}
    Let $G$ be a  connected linear algebraic group. 
    Let $S$ be a finite set of places of $k$ that contains at least one finite place and all archimedean places. 
    Then strong approximation for $BG$ off $S$ with respect to $\Br(BG)$ holds. 
\end{theorem}

\begin{proof}
    Fix a faithful representation 
    $$
    G\hookrightarrow \sl(V). 
    $$
    Let $X=G\setminus \sl(V)$ which is smooth. Then there is a presentation 
$$
f: X \to [X/\sl(V)] \cong BG. 
$$
The isomorphism is from Lemma \ref{l:little}. We will abuse notation and write $f$ for the map induced on $R$-points for various rings $R$.

   Now consider a point $\bar{x}\in BG(\bA_{S})^{\Br(BG)}$ which can be lifted to a point $x\in X(\bA_{S})^{\Br(X)}$, using 
   Corollary \ref{c:brIsom} and Proposition \ref{p:projTriv}. 
   
   Let $U$ be an adelic neighbourhood  of $\bar{x}$ so that $f^{-1}(U)$ is a neighbourhood of $x$. 
   Observe that strong approximation holds for $\sl(V)$, see \cite{Kenser1965} and \cite{Platonov1969}. So we may apply
   \cite[Theorem 6.1]{borovoi2013}. 
   By this theorem we can find a rational point $t\in X(k)$ and a point $g\in \sl(V)(k_S)$ so that $g.t$ belongs to $f^{-1}(U)$. 
   Now, $f(t)=f(gt)\in U$. Further, although $gt$ is not $k$-rational, we have that $f(gt)=f(t)\in U$ is. Hence the result.   
\end{proof}

\section*{Declarations}

The author has no conflicts of interest. 

No data was used in creating this manuscript. 

\printbibliography
\end{document}